\documentclass[11pt,letterpaper,reqno]{amsart}
\usepackage{fullpage}
\usepackage[top=2cm, bottom=4.5cm, left=2.5cm, right=2.5cm]{geometry}
\usepackage{amsmath,amsthm,amsfonts,amssymb,amscd}
\usepackage{geometry,tikz,tikz-cd}
\usepackage{empheq}
\usepackage{lipsum}
\usepackage{lastpage}
\usepackage{thmtools}
\usepackage{enumerate}
\usepackage[shortlabels]{enumitem}
\usepackage{fancyhdr}
\usepackage{mathrsfs}
\usepackage{xcolor}
\usepackage{graphicx}
\usepackage{listings}
\usepackage{bbm}
\usepackage{array}
\usepackage{hyperref}
\usepackage[makeroom]{cancel}

\usepackage[utf8]{inputenc}
\usepackage{comment}

\usepackage{makecell}

\usepackage[all]{xy}
\xyoption{matrix}
\xyoption{arrow}

\usetikzlibrary{arrows,decorations.pathmorphing,backgrounds,positioning,fit,petri}

\tikzset{help lines/.style={step=#1cm,very thin, color=gray},
help lines/.default=.5} 
\tikzset{thick grid/.style={step=#1cm,thick, color=gray},
thick grid/.default=1} 

\mathtoolsset{showonlyrefs=true}
\numberwithin{figure}{section}
\numberwithin{table}{section}

\theoremstyle{definition}

\theoremstyle{plain}
\newcommand{\thistheoremname}{}
\newtheorem*{genericthm*}{\thistheoremname}
\newenvironment{namedthm*}[1]
  {\renewcommand{\thistheoremname}{#1}%
   \begin{genericthm*}}
  {\end{genericthm*}}

\hypersetup{%
  colorlinks=true,
  linkcolor=blue,
  citecolor=blue,
  linkbordercolor={0 0 1}
}

\DeclareMathOperator{\ZZ}{\mathbb{Z}}

\newcommand{\lrabs}[1]{\left\lvert #1 \right\lvert}
\newcommand{\lrp}[1]{\left(#1\right)}
\newcommand{\lrb}[1]{\left[#1\right]}

\allowdisplaybreaks

\newtheorem{theorem}{Theorem}[section]
\newtheorem{corollary}[theorem]{Corollary}
\newtheorem{lemma}[theorem]{Lemma}

\newtheorem{conjecture}[theorem]{Conjecture}
\theoremstyle{definition}

\theoremstyle{definition}

\theoremstyle{remark}

\numberwithin{equation}{section}

\usepackage{latexsym}
\usepackage{amssymb}

\newcommand{\ben}{\begin{equation}}
\newcommand{\een}{\end{equation}}

\makeatletter
\NewDocumentCommand{\sump}{e{_}}
 {%
  \DOTSB
  \mathop{\IfNoValueTF{#1}{\sump@{}}{\sump@{#1}}}%
  \nolimits
 }
\newcommand{\sump@}[1]{\mathpalette\sump@@{#1}}
\newcommand{\sump@@}[2]{%
  \ifx#1\displaystyle
    {\sump@display{#2}}%
  \else
    \sum@\nolimits'_{#2}%
  \fi
}
\newcommand{\sump@display}[1]{%
  \sbox\z@{$\m@th\displaystyle\sum@\nolimits'$}%
  \sbox\tw@{$\m@th\displaystyle\sum@\limits_{#1}$}%
  \sbox\@tempboxa{$\m@th\displaystyle'$}
  \mathop{\sum@\nolimits' \kern-\wd\@tempboxa}\limits_{#1}%
  \ifdim\wd\z@>\wd\tw@
    \kern\dimexpr\wd\z@-\wd\tw@\relax
  \fi
}
\makeatother

\newcommand{\nw}{{\textnormal{new}}}

\DeclareMathOperator{\Tr}{Tr}

\DeclareMathOperator{\sgn}{sgn}
\setlist[enumerate]{leftmargin=*,widest=0}
\setlist[itemize]{leftmargin=*,widest=0}
\makeatletter
\def\subsection{\@startsection{subsection}{2}%
  \z@{.5\linespacing\@plus.7\linespacing}{.3\linespacing}%
  {\normalfont\bfseries}}

\def\subsubsection{\@startsection{subsubsection}{3}%
  \z@{.5\linespacing\@plus.7\linespacing}{.3\linespacing}%
  {\normalfont\bfseries}}
\makeatother

\begin{document}

\author[E. Ross]{Erick Ross}
\address[E. Ross]{School of Mathematical and Statistical Sciences, Clemson University, Clemson, SC, 29634}
\email{erickr@clemson.edu}

\author[H. Xue]{Hui Xue}
\address[H. Xue]{School of Mathematical and Statistical Sciences, Clemson University, Clemson, SC, 29634}
\email{huixue@clemson.edu}

\title{Asymptotics and sign patterns of Hecke polynomial coefficients}


\begin{abstract}
    We determine the asymptotic behavior of the coefficients of Hecke polynomials. In particular, this allows us to determine signs of these coefficients when the level or the weight is sufficiently large. In all but finitely many cases, this also verifies a conjecture on the nanvanishing of the coefficients of Hecke polynomials.
\end{abstract}

\maketitle

\section{Introduction} 

For integers $m \geq 1$, $N$ coprime to $m$, and $k\geq 2$ even, let $S_k(\Gamma_0(N))$ denote the space of cuspforms of level $N$ and weight $k$. 
Let $T'_{m}(N,k) := \frac{1}{m^{(k-1)/2}} T_{m}(N,k)$ denote the normalized $m$-th Hecke operator on $S_k(\Gamma_0(N))$. 
For each integer $r\ge0$, let $c_{r}(m,N,k)$ denote the $r$-th coefficient of the characteristic polynomial $T'_{m}(N,k)(x)$ associated to $T'_{m}(N,k)$ as follows:
\begin{align}
    T'_{m}(N,k)(x)=\sum_{r=0}^{d} c_r(m,N,k) x^{d-r},
\end{align}
where $d=\dim S_k(\Gamma_0(N))$. 
Hecke operators are of central importance in the theory of modular forms, and are completely characterized by the Hecke polynomials. We would like to study the coefficients of these Hecke polynomials in order to understand their structure.
In particular, for any fixed $m$ and $r$, the main goal of this paper is to determine the asymptotic behavior of $c_{r}(m,N,k)$ as $N+k \to \infty$. This will also show that $c_{r}(m,N,k)$ is nonvanishing and further determine its sign in all but finitely many cases.

We give an outline of this paper. In Section \ref{sec:prelim-calcs}, we apply the Girard-Newton formula to the coefficients $c_{r}(m,N,k)$, and state the asymptotic behavior of $\Tr T'_m(N,k)$.
In Section \ref{sec:m-sq}, we consider the case when $m$ is a perfect square and prove the following result determining the asymptotic behavior of the $c_r(m,N,k)$. In the following, all big-$O$ notation is with respect to $N$ and $k$. Additionally, we use the notation "$O(N^\varepsilon)$", for example, to mean "$O(N^\varepsilon)$ for all $\varepsilon > 0$". 
\begin{theorem} \label{thm:m-sq}
    Fix an integer $r \ge 0$ and a perfect square $m \ge 1$.
    Then for $N$ coprime to $m$ and $k \geq 2$ even,
    \begin{align}
        c_r(m,N,k) = \frac{(-1)^r}{r!} \lrp{ \frac{1}{\sqrt{m}} \frac{k-1}{12} \psi(N)}^r + O(k^{r-1}N^{r-1/2+\varepsilon}).
    \end{align}
    Here, $\psi(N)$ denotes the multiplicative function $\psi(N) = N \prod_{p\mid N} \lrp{1+\frac1p}$.
\end{theorem}

In Section \ref{sec:m-nsq}, we consider the case when $m$ is not a perfect square and establish the following asymptotics  of $c_r(m,N,k)$. Recall here that $\sigma_1(m)$ denotes the sum of divisors function $\sigma_1(m) := \sum_{d \mid m} d$, and that $(2r)!!$ denotes the double factorial $(2r)!! := 2r (2r-2)(2r-4) \cdots 2$.
\begin{theorem} \label{thm:m-nsq}
    Fix an integer $r \ge 0$ and a non-square $m \ge 1$.
    Then for $N$ coprime to $m$ and $k \geq 2$ even,
    \begin{align}
        c_{2r}(m,N,k) &= \frac{(-1)^r}{(2r)!!} \lrp{\frac{\sigma_1(m)}{m} \frac{k-1}{12} \psi(N)}^r + O(k^{r-1} N^{r-1/2+\varepsilon}) \qquad \text{and}  \\
        c_{2r+1}(m,N,k) &= c_1(m,N,k) \cdot \frac{(-1)^r}{(2r)!!}  \lrp{\frac{\sigma_1(m)}{m} \frac{k-1}{12}\psi(N)}^r  + O(k^{r-1} N^{r-1/2+\varepsilon}). 
    \end{align}
\end{theorem}

In Section \ref{sec:new-subspace}, we extend Theorems \ref{thm:m-sq} and \ref{thm:m-nsq} to the new subspace $S_k^\nw(\Gamma_0(N))$.

Finally, in Sections \ref{sec:sign-patterns-more-general} and \ref{sec:conj-coeffs}, we discuss these results. In Section \ref{sec:sign-patterns-more-general}, we discuss how the arguments given in Theorems \ref{thm:m-sq} and \ref{thm:m-nsq} for the Hecke polynomials can also be applied to other polynomials. In particular, these arguments reveal a coefficient sign pattern for a wide class of polynomials. 
Then in Section \ref{sec:conj-coeffs}, we discuss a conjecture on the nonvanishing of the Hecke polynomial coefficients and survey its current progress.

\section{Preliminary Calculations} \label{sec:prelim-calcs}

For simplicity, we write $c_r$ for the coefficients $c_r(m,N,k)$. Let $d=\dim S_k(\Gamma_0(N))$ and $\lambda_1, \ldots, \lambda_d$ denote the eigenvalues of $T'_m(N,k)$. Observe that $(-1)^r c_r$ is just the $r$-th elementary symmetric polynomial of these eigenvalues:
\begin{align}
    c_0=1,\quad -c_1 = \sum_{1\leq i\leq d} \lambda_i \,, \quad c_2 = \sum_{1\leq i < j \leq d} \lambda_i \lambda_j \,, \quad -c_3 = \sum_{1\leq i<j<\ell\leq d} \lambda_i \lambda_j \lambda_\ell\,, \quad \ldots
\end{align}
We also write $p_r$ for the sum of $r$-th powers of these eigenvalues:
\begin{align} \label{eqn:pr-def}
    p_r:=\sum_{i=1}^{d} \lambda_i^r.
\end{align}
Then the Girard-Newton identities yield the following relation
between the $c_r$ and the $p_r$.
\begin{lemma}[{\cite[p. 38]{tignol}}] \label{lem:newton-id}
Let $c_r$ and $p_r$ be defined as above. Then for $r\ge1$,
    \begin{align}
        c_r = \frac{-1}{r} \sum_{j=1}^r c_{r-j} p_j.
    \end{align}
\end{lemma}

We also give estimates on the traces of Hecke operators. These estimates will be needed shortly when we express the $p_r$ in terms of traces of certain Hecke operators.
In a previous paper \cite{ross-xue}, we proved the following result by analyzing the various terms of the Eichler-Selberg trace formula.  
\begin{lemma}[{\cite[Lemmas 4.1, 4.2]{ross-xue}}] \label{lem:trace-estimate}
Fix an integer $m \ge 1$. Then for $N$ coprime to $m$ and $k \geq 2$ even,
\begin{align}
    \Tr T'_m(N,k) = \begin{dcases}
        \frac{1}{\sqrt{m}}\frac{k-1}{12} \psi(N)  + O(N^{1/2+\varepsilon}),  & \text{if $m$ is a perfect square}, \\
        O(N^\varepsilon), & \text{if $m$ is not a perfect square}.
    \end{dcases}
\end{align}
\end{lemma}
To gauge the growth of the terms in this formula, note that $\psi(N) \geq N$ and $\psi(N) = O(N^{1+\varepsilon})$ \cite[Sections~18.1,~22.13]{hardy-wright}.

\section{When \texorpdfstring{$m$}{m} is a perfect square} \label{sec:m-sq}

In this section, we consider the  case when $m$ is a perfect square.
We then have the following estimates on the $p_j$ \eqref{eqn:pr-def}.
\begin{lemma} \label{lem:m-sq:pj-estimates}
Fix an integer $r \ge 1$ and a perfect square $m \ge 1$.
Then for $N$ coprime to $m$ and $k \geq 2$ even,
\begin{align} 
    p_1 &= \frac{1}{\sqrt{m}} \frac{k-1}{12} \psi(N)  + O(N^{1/2+\varepsilon}), \qquad \text{and} \\
    p_{j} &= O(kN^{1+\varepsilon}),  \qquad \text{for all $1 \leq j \leq r$}.
\end{align}
\end{lemma}
\begin{proof}
    The first claim follows immediately from Lemma \ref{lem:trace-estimate}. 
    
    For the second claim, note from Lemma \ref{lem:trace-estimate} and the fact that $\psi(N)=O(N^{1+\epsilon})$,
    \begin{align}
        d := \dim S_k(\Gamma_0(N)) = \Tr T'_1 = O(kN^{1+\varepsilon}).
    \end{align}
    Then utilizing Deligne's bound $\lrabs{\lambda_i} \leq \sigma_0(m) = \sum_{d\mid m} 1$, we obtain
    \begin{align}
        \lrabs{p_j} = \lrabs{\sum_{i=1}^d \lambda_i^j} \leq \sum_{i=1}^d \sigma_0(m)^j = O(kN^{1+\varepsilon}),
    \end{align}
    as desired.
\end{proof}

For $m$ and $r$ fixed, we now determine the asymptotic behavior of $c_r(m,N,k)$ as $N+k \to \infty$. Note $c_r(m,N,k)$ is not technically defined for $N,k$ such that $\dim S_k(\Gamma_0(N)) < r$. However, there are only finitely many such pairs $(N,k)$ \cite[Theorem 1.1]{ross}, so it is well-defined here to ask about $c_r(m,N,k)$ as $N+k\to\infty$.

{
\renewcommand{\thetheorem}{\ref{thm:m-sq}}
\begin{theorem}
     Fix an integer $r \ge 0$ and a perfect square $m \ge 1$.
     Then for $N$ coprime to $m$ and $k \geq 2$ even,
    \begin{align}
        c_r(m,N,k) = \frac{(-1)^r}{r!} \lrp{ \frac{1}{\sqrt{m}} \frac{k-1}{12} \psi(N)}^r + O(k^{r-1}N^{r-1/2+\varepsilon}).
    \end{align}
\end{theorem}
\addtocounter{theorem}{-1}
}
\begin{proof}
    We proceed by strong induction on $r$. The base case of $r=0$ is immediate since $c_0 = 1$. 

    For $r \geq 1$, we have by Lemma \ref{lem:newton-id} that 
    \begin{align} \label{eq:cr-induction}
        c_{r} = \frac{-1}{r} \sum_{j=1}^{r} c_{r-j} p_j 
        = \frac{-1}{r} \Bigg[ c_{r-1} p_1 + \sum_{j=2}^{r} c_{r-j} p_j  \Bigg].
    \end{align}
    Then by the induction hypothesis,
    \begin{align}
        c_{r-1} &= \frac{(-1)^{r-1}}{(r-1)!} \lrp{ \frac{1}{\sqrt{m}} \frac{k-1}{12} \psi(N)}^{r-1} + O(k^{r-2}N^{r-3/2+\varepsilon}),
        \\
        c_{r-j} &= O(k^{r-2} N^{r-2+\varepsilon}), \qquad\qquad\qquad \text{for $2\le j \le r$},
    \end{align}
    and by Lemma \ref{lem:m-sq:pj-estimates}, 
    \begin{align}
        p_1 &= \frac{1}{\sqrt{m}} \frac{k-1}{12} \psi(N) + O(N^{1/2+\varepsilon}), \\
        p_j &=  O(kN^{1+\varepsilon}), \qquad\qquad\qquad \text{for $2\le j \le r$}.
    \end{align}

    Applying these estimates to \eqref{eq:cr-induction}, we obtain
    \begin{align}
        c_{r} &=  \frac{-1}{r}  \Bigg[ c_{r-1} p_1 + \sum_{j=2}^{r} c_{r-j} p_j \Bigg] \\
        &= \frac{-1}{r} \Bigg[   \lrp{\frac{(-1)^{r-1}}{(r-1)!} \lrp{\frac{1}{\sqrt{m}} \frac{k-1}{12} \psi(N)}^{r-1} + O(k^{r-2} N^{r-3/2+\varepsilon}) } \\
        & \qquad\qquad\qquad\qquad \times \lrp{ \frac{1}{\sqrt{m}} \frac{k-1}{12} \psi(N) + O(N^{1/2+\varepsilon}) }  \\
        &\qquad\ \ \ \  +  \sum_{j=2}^{r}  O(k^{r-2} N^{r-3/2+\varepsilon}) \cdot O(k N^{1+\varepsilon}) \Bigg] \\
        &= \frac{-1}{r} \Bigg[ \frac{(-1)^{r-1}}{(r-1)!} \lrp{\frac{1}{\sqrt m} \frac{k-1}{12} \psi(N)}^{r} + O(k^{r-1} N^{r-1/2+\varepsilon})  \Bigg] \\
        &= \frac{(-1)^r}{r!} \lrp{\frac{1}{\sqrt m} \frac{k-1}{12} \psi(N)}^{r} + O(k^{r-1} N^{r-1/2+\varepsilon}).
    \end{align}
    This completes the proof. 
\end{proof}

Theorem \ref{thm:m-sq} allows us in particular to determine the sign of $c_r(m,N,k)$ for all but finitely many pairs $(N,k)$.
\begin{corollary} \label{cor:sgn-msq}
    Fix an integer $r \ge 0$ and a perfect square $m \ge 1$.
    Then $c_{r}(m,N,k)$ has sign $(-1)^r$ for all but finitely pairs $(N,k)$. 
\end{corollary}
\begin{proof}
    Since $\psi(N) \ge N$, we can write the asymptotic formula from Theorem \ref{thm:m-sq} as 
    \begin{align}
        c_r(m,N,k) &= \frac{(-1)^r}{r! \sqrt{m}^r} \lrp{\frac{k-1}{12} \psi(N)}^r \lrb{1 
        +  O\lrp{k^{r-1}N^{r-1/2+\varepsilon}}  \lrp{\frac{k-1}{12} \psi(N)}^{-r} } \\
        &= \frac{(-1)^r}{r! \sqrt{m}^r} \lrp{\frac{k-1}{12} \psi(N)}^r \lrb{
        1 + O\lrp{k^{-1}N^{-1/2+\varepsilon}} 
        }.
    \end{align}
    Then since the $O\lrp{k^{-1}N^{-1/2+\varepsilon}}$ term tends to $0$ as $N \to \infty$ or $k \to \infty$, it will have magnitude less than $1$ for all but finitely many pairs $(N,k)$. This then yields the desired result.
\end{proof}

\section{When \texorpdfstring{$m$}{m} is not a perfect square} \label{sec:m-nsq}

In this section, we consider the remaining case when $m$ is not a perfect square.
First, we have the following estimates on the $p_j$ \eqref{eqn:pr-def}.
\begin{lemma} \label{lem:m-nsq:pj-estimates}
Fix an integer $r \ge 0$ and a non-square $m \ge 1$.
Then
\begin{align}
    p_1 &= -c_1 = O(N^\varepsilon), \\
    p_2 &= \frac{\sigma_1(m)}{m} \frac{k-1}{12} \psi(N) + O(N^{1/2+\varepsilon}), \\
    p_{3} &= O(N^\varepsilon), \\
    p_{j} &= O(kN^{1+\varepsilon}),  \qquad \text{for all $1 \leq j \leq r$}.
\end{align}
\end{lemma}
\begin{proof}
    The first claim follows immediately from Lemma \ref{lem:trace-estimate}. 

    For the second claim, observe that $p_2 = \Tr {T_m'}^{2}$. Then by the Hecke operator composition formula \cite[Theorem 10.2.9]{cohen-stromberg} and Lemma \ref{lem:trace-estimate},
    \begin{align}
        p_2 = \Tr {T_m'}^2 
        = \sum_{d \mid m} \Tr T_{m^2/d^2}' 
        &= \sum_{d \mid m} \frac{d}{m} \frac{k-1}{12} \psi(N) + O(N^{1/2+\varepsilon}) \\
        &= \frac{\sigma_1(m)}{m} \frac{k-1}{12} \psi(N) + O(N^{1/2+\varepsilon}).
    \end{align}

    For the third claim, we similarly have by the Hecke operator composition formula and Lemma \ref{lem:trace-estimate},
    \begin{align}
        p_3 = \Tr {T_m'}^3 
        = \Tr \sum_{d \mid m} T_{m^2/d^2}' T_m' 
        &= \sum_{d \mid m} \ \sum_{\delta \mid (m^2/d^2, m)} \Tr T'_{m^3/d^2\delta^2} \\
        &= \sum_{d \mid m} \ \sum_{\delta \mid (m^2/d^2, m)} O(N^\varepsilon) \\
        &= O(N^{\varepsilon}).
    \end{align}

    Finally, the fourth claim follows from an identical argument as in Lemma \ref{lem:m-sq:pj-estimates}.
\end{proof}

For $m$ and $r$ fixed, we now determine the asymptotic behavior of $c_r(m,N,k)$ as $N+k \to \infty$.

{
\renewcommand{\thetheorem}{\ref{thm:m-nsq}}
\begin{theorem}
    Fix an integer $r \ge 0$ and a non-square $m \ge 1$. Then for $N$ coprime to $m$ and $k \geq 2$ even,
    \begin{align}
        c_{2r} &= \frac{(-1)^r}{(2r)!!} \lrp{\frac{\sigma_1(m)}{m} \frac{k-1}{12} \psi(N)}^r + O(k^{r-1} N^{r-1/2+\varepsilon}) \qquad \text{and}  \\
        c_{2r+1} &= c_1 \cdot \frac{(-1)^r}{(2r)!!}  \lrp{\frac{\sigma_1(m)}{m} \frac{k-1}{12}\psi(N)}^r  + O(k^{r-1} N^{r-1/2+\varepsilon}).  \label{eqn:main-thm-2r1}
    \end{align}
\end{theorem}
\addtocounter{theorem}{-1}
}
\begin{proof}
    We proceed by strong induction on $r$. 
    The base case of $r=0$ is immediate since $c_0 = 1$ and $c_1 = c_1$. 

    Then for $r\ge1$, we have from Lemma \ref{lem:newton-id} that
    \begin{align} \label{eqn:c2r-sum-temp}
        c_{2r} &= \frac{-1}{2r} \sum_{j=1}^{2r} c_{2r-j} p_j = \frac{-1}{2r} \Bigg[ c_{2r-1} p_1 + c_{2r-2} p_2 + \sum_{j=3}^{2r} c_{2r-j} p_j  \Bigg].
    \end{align}
    Then by the induction hypotheses,
    \begin{align}
        c_{2r-1} &= O(k^{r-1} N^{r-1+\varepsilon}), \\
        c_{2r-2} &= \frac{(-1)^{r-1}}{(2r-2)!!} \lrp{\frac{\sigma_1(m)}{m} \frac{k-1}{12} \psi(N)}^{r-1} + O(k^{r-2} N^{r-3/2+\varepsilon}), \\
        c_{2r-j} &= O(k^{r-2} N^{r-2+\varepsilon}), \qquad\qquad \text{for $3\le j \le 2r$},
    \end{align}
    and by Lemma \ref{lem:m-nsq:pj-estimates}, 
    \begin{align}
        p_1 &= O(N^\varepsilon), \\
        p_2 &= \frac{\sigma_1(m)}{m} \frac{k-1}{12} \psi(N) + O(N^{1/2+\varepsilon}), \\
        p_j &=  O(kN^{1+\varepsilon}), \qquad\qquad \text{for $3\le j \le 2r$}.
    \end{align}

    Applying these estimates to \eqref{eqn:c2r-sum-temp}, we obtain
    \begin{align}
        c_{2r}  &=  \frac{-1}{2r}  \Bigg[ c_{2r-1} p_1 +  c_{2r-2} p_2  + \sum_{j=3}^{2r} c_{2r-j} p_j \Bigg] \\
        &= \frac{-1}{2r} \Bigg[ O(k^{r-1} N^{r-1+\varepsilon}) \cdot O(N^\varepsilon) \\
        &\qquad \  +  \lrp{\frac{(-1)^{r-1}}{(2r-2)!!} \lrp{\frac{\sigma_1(m)}{m} \frac{k-1}{12} \psi(N)}^{r-1} + O(k^{r-2} N^{r-3/2+\varepsilon}) } \\
        & \qquad\qquad\qquad\qquad\qquad\qquad \times \lrp{ \frac{\sigma_1(m)}{m} \frac{k-1}{12} \psi(N) + O(N^{1/2+\varepsilon}) }  \\
        &\qquad \ +  \sum_{j=3}^{2r}  O(k^{r-2} N^{r-2+\varepsilon}) \cdot O(k N^{1+\varepsilon}) \Bigg] \\
        &= \frac{-1}{2r} \Bigg[ O(k^{r-1}N^{r-1+\varepsilon})  + \frac{(-1)^{r-1}}{(2r-2)!!} \lrp{\frac{\sigma_1(m)}{m} \frac{k-1}{12} \psi(N)}^{r} \\
        &\qquad\qquad\qquad\qquad + O(k^{r-1} N^{r-1/2+\varepsilon}) + O(k^{r-1} N^{r-1+\varepsilon}) \Bigg] \\
        &= \frac{(-1)^r}{(2r)!!} \lrp{\frac{\sigma_1(m)}{m} \frac{k-1}{12} \psi(N)}^{r} + O(k^{r-1} N^{r-1/2+\varepsilon}),
    \end{align}
    verifying the first claim of the inductive step. 

    For the second claim of the inductive step, we similarly have
    from Lemma \ref{lem:newton-id} that
    \begin{align} \label{eqn:c2r1-sum-temp}
        c_{2r+1} &= \frac{-1}{2r+1} \sum_{j=1}^{2r+1}  c_{2r+1-j} p_j = \frac{-1}{2r+1} \Bigg[ c_{2r} p_1 + c_{2r-1} p_2 +  c_{2r-2} p_3 + \sum_{j=4}^{2r+1} c_{2r+1-j} p_j  \Bigg].
    \end{align}
    Then by the induction hypotheses and the proof for $c_{2r}$, 
    \begin{align}
        c_{2r} &= \frac{(-1)^r}{(2r)!!} \lrp{\frac{\sigma_1(m)}{m} \frac{k-1}{12} \psi(N)}^{r} + O(k^{r-1} N^{r-1/2+\varepsilon}), \\
        c_{2r-1} &= c_1 \cdot  \frac{(-1)^{r-1}}{(2r-2)!!}  \lrp{\frac{\sigma_1(m)}{m} \frac{k-1}{12}\psi(N)}^{r-1}  + O(k^{r-2} N^{r-3/2+\varepsilon}), \\
        c_{2r-2} &= O(k^{r-1} N^{r-1+\varepsilon}), \\
        c_{2r+1-j} &= O(k^{r-2} N^{r-2+\varepsilon}), \qquad\qquad \text{for $4\le j \le 2r+1$},
    \end{align}
    and by Lemma \ref{lem:m-nsq:pj-estimates}, 
    \begin{align}
        p_1 &= -c_1, \\
        p_2 &= \frac{\sigma_1(m)}{m} \frac{k-1}{12} \psi(N) + O(N^{1/2+\varepsilon}), \\
        p_3 &= O(N^\varepsilon), \\
        p_j &=  O(kN^{1+\varepsilon}), \qquad\qquad \text{for $4\le j \le 2r+1$}.
    \end{align}
    Applying these estimates to \eqref{eqn:c2r1-sum-temp}, we obtain
    \begin{align}
        c_{2r+1} &= \frac{-1}{2r+1} \Bigg[ c_{2r} p_1 + c_{2r-1} p_2 + c_{2r-2} p_3  + \sum_{j=4}^{2r+1} c_{2r+1-j} p_j \Bigg] \\
        &= \frac{-1}{2r+1} \Bigg[\lrp{\frac{(-1)^r}{(2r)!!} \lrp{\frac{\sigma_1(m)}{m} \frac{k-1}{12} \psi(N)}^{r} + O(k^{r-1} N^{r-1/2+\varepsilon}) } \cdot (-c_1) \\
        &\ \ \ + \lrp{c_1 \cdot  \frac{(-1)^{r-1}}{(2r-2)!!}  \lrp{\frac{\sigma_1(m)}{m} \frac{k-1}{12}\psi(N)}^{r-1}  + O(k^{r-2} N^{r-3/2+\varepsilon})} \\
        &\qquad\qquad\qquad\qquad\qquad\qquad  \times \lrp{\frac{\sigma_1(m)}{m} \frac{k-1}{12} \psi(N) + O(N^{1/2+\varepsilon})} \\
        &\ \ \ + O(k^{r-1} N^{r-1+\varepsilon}) \cdot O(N^\varepsilon)  + \sum_{j=4}^{2r+1}   O(k^{r-2} N^{r-2+\varepsilon}) \cdot O(kN^{1+\varepsilon}) \Bigg] \\
        &= \frac{-1}{2r+1} \Bigg[ -c_1 \cdot \frac{(-1)^r}{(2r)!!} \lrp{\frac{\sigma_1(m)}{m} \frac{k-1}{12} \psi(N)}^{r}  + O(k^{r-1} N^{r-1/2+\varepsilon})   \\
        &\ \ \ + c_1 \cdot  \frac{(-1)^{r-1}}{(2r-2)!!}  \lrp{\frac{\sigma_1(m)}{m} \frac{k-1}{12}\psi(N)}^{r}  + O(k^{r-1} N^{r-1/2+\varepsilon})  \\
        &\ \ \ + O(k^{r-1} N^{r-1+\varepsilon})  +  O(k^{r-1} N^{r-1+\varepsilon}) \Bigg]  \\
        &= c_1 \cdot \frac{1}{2r+1} \lrp{\frac{(-1)^r}{(2r)!!} - \frac{(-1)^{r-1}}{(2r-2)!!} } \cdot \lrp{\frac{\sigma_1(m)}{m} \frac{k-1}{12} \psi(N)}^{r}  + O(k^{r-1} N^{r-1/2+\varepsilon})  \\
        &= c_1 \cdot  \frac{(-1)^r}{(2r)!!}  \lrp{\frac{\sigma_1(m)}{m} \frac{k-1}{12} \psi(N)}^{r}  + O(k^{r-1} N^{r-1/2+\varepsilon}),
    \end{align}
    verifying the second claim of the inductive step.

    This completes the proof.
\end{proof}

Theorem \ref{thm:m-nsq} allows us in particular to determine the sign of the even-indexed coefficients for all but finitely many pairs $(N,k)$. The following corollary can be shown using an identical argument as in Corollary \ref{cor:sgn-msq}.
\begin{corollary} \label{cor:sgn-mnsq-even} 
    Fix an integer $r \ge 0$ and a non-square $m \ge 1$.
    Then $c_{2r}(m,N,k)$ has sign $(-1)^r$ for all but finitely pairs $(N,k)$. 
\end{corollary}

The behavior of  the odd-indexed coefficients, on the other hand, is determined by the behavior of the trace. 
\begin{corollary} \label{cor:sgn-mnsq-odd}
    Fix an integer $r \ge 0$, a non-square $m \ge 1$, and an even integer $k \ge 2$. Consider $N$ such that $\Tr T'_m(N,k) \neq 0$. Then $c_{2r+1}(m,N,k)$ has sign $(-1)^{r+1} \sgn(\Tr T'_m(N,k))$ for all but finitely many $N$. 
\end{corollary}
\begin{proof}
    Since $\psi(N) \ge N$, we can write the asymptotic formula from Theorem \ref{thm:m-nsq} as

    \begin{align}
        c_{2r+1} &= \frac{(-1)^r}{(2r)!!}  \lrp{\frac{\sigma_1(m)}{m}}^r \lrp{ \frac{k-1}{12}\psi(N)}^r \lrb{c_1 + O\lrp{k^{r-1} N^{r-1/2+\varepsilon}} \lrp{ \frac{k-1}{12}\psi(N)}^{-r} }  \\
        &= \frac{(-1)^r}{(2r)!!}  \lrp{\frac{\sigma_1(m)}{m}}^r \lrp{ \frac{k-1}{12}\psi(N)}^r \lrb{c_1 + O\lrp{k^{-1} N^{-1/2+\varepsilon}} } 
    \end{align}
    Then observe that since $\Tr T_m \in \ZZ$ and $c_1 = -\Tr T'_m = - m^{-(k-1)/2} \Tr T_m \neq 0$, we must have $\lrabs{c_1} = \lrabs{\Tr T'_m} \geq m^{-(k-1)/2}$.
    And because the $O\lrp{k^{-1}N^{-1/2+\varepsilon}}$ term tends to $0$ as $N \to \infty$, it will have magnitude less than $m^{-(k-1)/2}$ for all but finitely many $N$. This yields the desired result.
\end{proof}

We note that the condition $\Tr T'_m(N,k) \neq 0$ here is not overly restrictive. Rouse \cite[Theorem 1.6]{rouse} showed that there are only finitely many $k$ for which we could possibly have $\Tr T'_m(N,k) = 0$ for some $N$.  And even for these finitely many remaining $k$, he showed in \cite[Theorem 1.7]{rouse} that $\Tr T'_m(N,k) \neq 0$ for $100\%$ of $N$.
He further conjectured in \cite[Conjecture 1.5]{rouse} that $\Tr T'_m(N,k) \neq 0$ for all $N \ge 1$ and $k=12$ or $\geq 16$. 

We also note that there are various ways one could try to improve this result to where $N$ and $k$ both vary. The only reason we fixed $k$ was to guarantee that $c_1 = -\Tr T'_m(N,k)$ was bounded away from $0$. If we relax the condition of $k$ being fixed to just that $k \leq (1-\delta)\log_m(N)$ for some $\delta > 0$, then we have the same result for all but finitely many pairs $(N,k)$. One could also try to bound $\Tr T'_m(N,k)$ away from $0$ using some sort of vertical Atkin-Serre type result for the trace (e.g. along the lines of \cite[Theorem 2.2]{kim}).

\section{Extending to the new subspace} \label{sec:new-subspace}

All of our results extend to the Hecke polynomial over the new subspace. 
Let $T'^{\,\nw}_m(N,k)$ denote the restriction of $T'_m(N,k)$ to the new subspace $S_k^\nw(\Gamma_0(N))$.
Let $c^\nw_{r}(m,N,k)$ denote the $r$-th coefficient of the characteristic polynomial $T'^{\,\nw}_{m}(N,k)(x)$  as follows:
\begin{align}
    T'^{\,\nw}_{m}(N,k)(x)=\sum_{r=0}^{d^\nw} c^\nw_r(m,N,k) x^{d^\nw-r}.
\end{align}
Here, $d^\nw=\dim S_k^\nw(\Gamma_0(N))$.
Just like before, we can then 
determine the asymptotic behavior of $c_r^\nw(m,N,k)$ as $N+k \to \infty$.
Note $c_r^\nw(m,N,k)$ is not technically defined for $N,k$ such that $\dim S_k^\nw(\Gamma_0(N)) < r$. However, there are only finitely many such pairs $(N,k)$ \cite[Theorem 1.3]{ross}, so it is perfectly well-defined here to ask about $c_r^\nw(m,N,k)$ as $N+k\to\infty$.

Let
\begin{equation}
    \psi^\nw(N) := \prod_{p^r \parallel N} \begin{cases}
        p\lrp{1-\frac1p}, & \text{if } r=1, \\
        p^2\lrp{1-\frac1p - \frac{1}{p^2}}, & \text{if } r=2, \\
        p^r\lrp{1-\frac1p-\frac{1}{p^2}+\frac{1}{p^3}}, & \text{if } r\geq3,
    \end{cases} 
\end{equation}
and note that $\psi^\nw(N) \leq N$ and $\psi^\nw(N) = \Omega(N^{1-\varepsilon})$ \cite[Sections~18.1,~22.13]{hardy-wright}. 
In \cite[Lemmas 4.2, 4.3]{cason-et-al}, Cason et al. showed that for fixed $m$, 
\begin{align}
    \Tr T'^{\,\nw}_m(N,k) = \begin{dcases}
        \frac{1}{\sqrt{m}} \frac{k-1}{12} \psi^\nw(N)  + O(N^{1/2}),  & \text{if $m$ is a perfect square}, \\
        O(N^\varepsilon), & \text{if $m$ is not a perfect square}.
    \end{dcases}
\end{align}

The following two theorems then follow by an identical argument as in Theorems \ref{thm:m-sq} and \ref{thm:m-nsq}. The details are omitted.

\begin{theorem} \label{thm:new-m-sq}
    Fix an integer $r \ge 0$ and a perfect square $m \ge 1$.
    Then for $N$ coprime to $m$ and $k \geq 2$ even,
    \begin{align}
        c^\nw_r(m,N,k) = \frac{(-1)^r}{r!} \lrp{ \frac{1}{\sqrt{m}} \frac{k-1}{12} \psi^\nw(N)}^r + O(k^{r-1}N^{r-1/2}).
    \end{align}
\end{theorem}

\begin{theorem} \label{thm:new-m-nsq}
    Fix an integer $r \ge 0$ and a non-square $m \ge 1$.
    Then
    \begin{align}
        c^\nw_{2r}(m,N,k) &= \frac{(-1)^r}{(2r)!!} \lrp{\frac{\sigma_1(m)}{m} \frac{k-1}{12} \psi^\nw(N)}^r + O(k^{r-1} N^{r-1/2}) \qquad \text{and}  \\
        c^\nw_{2r+1}(m,N,k) &= c^\nw_1(m,N,k) \cdot \frac{(-1)^r}{(2r)!!}  \lrp{\frac{\sigma_1(m)}{m} \frac{k-1}{12}\psi^\nw(N)}^r  + O(k^{r-1} N^{r-1/2+\varepsilon}). 
    \end{align}
\end{theorem}
Then just like in Corollaries \ref{cor:sgn-msq}, \ref{cor:sgn-mnsq-even}, and \ref{cor:sgn-mnsq-odd}, this tells us the sign of the $c^\nw_r$ in all but finitely many cases.

\begin{corollary}
    Fix an integer $r \ge 0$ and a perfect square $m \ge 1$.
    Then $c^\nw_{r}(m,N,k)$ has sign $(-1)^r$ for all but finitely pairs $(N,k)$. 
\end{corollary}

\begin{corollary} 
    Fix an integer $r \ge 0$ and a non-square $m \ge 1$.
    Then $c^\nw_{2r}(m,N,k)$ has sign $(-1)^r$ for all but finitely pairs $(N,k)$. 
\end{corollary}

\begin{corollary}  
    Fix an integer $r \ge 0$, a non-square $m \ge 1$, and an even integer $k \ge 2$. Consider $N$ such that $\Tr T'^{\,\nw}_m(N,k) \neq 0$. Then $c^\nw_{2r+1}(m,N,k)$ has sign $(-1)^{r+1} \sgn(\Tr T'^{\,\nw}_m(N,k))$ for all but finitely many $N$. 
\end{corollary}

\section{Sign patterns for more general polynomials} \label{sec:sign-patterns-more-general}
In response to our previous paper showing that $c_2$ tends to be negative \cite{ross-xue}, Kimball Martin suggested to us that $c_2$ might display a similar bias more generally for polynomials with totally real roots. In fact, the sign tendencies for \textit{all} the coefficients given in Corollaries \ref{cor:sgn-mnsq-even} and \ref{cor:sgn-mnsq-odd} hold more generally for a wide class of polynomials with totally real roots. Essentially the only two conditions we need to impose are that the roots are distributed over an interval $[-A,A]$ in a roughly symmetric way about the origin, and that the roots are not all clustered at the origin.

More precisely, for $A>0$ and $r$ fixed, consider a sequence of polynomials $f_n$ with totally real roots lying in the interval $[-A,A]$. Let $d_n$ denote the degree of $f_n$, and let the $c_j(n)$ and $p_j(n)$ be defined as above in Section \ref{sec:prelim-calcs}.
We assume that $p_1(n) = o(d_n^{1/3})$ and $p_3(n) = o(d_n)$ (which will occur if the roots $x_1, \ldots, x_{d_n}$ of $f_n$ are distributed in a roughly symmetric way about the origin).
Also note that $p_2(n) = \alpha_n^2 d_n$, where $\alpha_n$ denotes the quadratic mean of the roots of $f_n$. 
We assume that $\alpha_n$ is bounded away from $0$ (which will occur as long as the roots are not all clustered at the origin).
Finally, observe that $\lrabs{p_j(n)} \leq A^r  d_n = O(d_n)$ for each $1 \leq j \leq r$.
These estimates 
\begin{align}
    p_1(n) &= -c_1(n) = o(d_n^{1/3}), \\
    p_2(n) &= \alpha_n^2 d_n, \\
    p_{3}(n) &= o(d_n), \\
    p_{j}(n) &= O(d_n)  \qquad \text{for each $1 \leq j \leq r$},
\end{align}
are essentially the content of Lemma \ref{lem:m-nsq:pj-estimates} (except that Lemma \ref{lem:m-nsq:pj-estimates} has much stronger error bounds, and where in that case, $\alpha_{N,k}$ tends to $\sqrt{\sigma_1(m)/m}$ as $N+k \to \infty$ \cite[Theorem 1.1]{cason-et-al-avg-size}). 

Then using an identical argument as in Theorem \ref{thm:m-nsq}, one can show that
\begin{align}
    c_{2r}(n) &= \frac{(-1)^r}{(2r)!!} \lrp{\alpha_n^2 d_n}^r + o(d_n^{r-1/3}) \qquad \text{and}  \\
    c_{2r+1}(n) &= c_1(n) \frac{(-1)^r}{(2r)!!}  \lrp{\alpha_n^2 d_n}^r  + o(d_n^{r}).
\end{align}
In particular, this means that as $d_n \to \infty$, the coefficients of $f_n$ will tend to the sign pattern
\begin{align}
    &+--++--++--++\cdots \qquad \text{if $c_1$ is bounded below $0$}, \\
    &++--++--++--+\cdots \qquad \text{if $c_1$ is bounded above $0$}.
\end{align}

For example, browsing the polynomials with totally real roots of degree 10 in LMFDB, almost all of them follow this sign pattern; see \cite{lmfdb} (such polynomials given by LMFDB are shifted so that their roots are roughly symmetric about the origin).

If the roots of a polynomial are perfectly symmetric about the origin, then we will have $c_1 = 0$, and the sign pattern becomes
\begin{align}
    &+0-0+0-0+0-0+\cdots. \qquad \qquad \qquad\ \ \ \ \ ~
\end{align}
For example, the roots of the Chebyshev polynomials are distributed in $[-1,1]$ in a perfectly symmetric way, and their coefficients follow precisely this pattern.

We also note what happens when the roots of a polynomial are not distributed symmetrically about the origin. If all the roots have the same sign, then the coefficients $c_r$ follow the sign pattern
\begin{align}
    &+-+-+-+-\ldots \qquad \text{if all the roots are positive}, \\
    &+\hspace{0.4pt}+\hspace{0.4pt}+\hspace{0.4pt}+\hspace{0.4pt}+\hspace{0.4pt}+\hspace{0.4pt}+\hspace{0.4pt}+\hspace{0.4pt} \ldots \qquad \text{if all the roots are negative}. 
\end{align}
When $m$ is a perfect square, most of the roots of $T'_m(N,k)(x)$ are positive, and Corollary \ref{cor:sgn-msq} shows that the coefficients of $T'_m(N,k)(x)$ tend to this first pattern.

\section{A conjecture on Hecke polynomial coefficients} \label{sec:conj-coeffs}

In \cite[Conjecture 1.5]{rouse}, Rouse gave the generalized Lehmer conjecture: that for all $m\geq1$, $N$ coprime to $m$, and $k = 12$ or $\geq 16$, $\Tr T_m(N,k) \neq 0$. More recently, Clayton et al. \cite[Conjecture 5.1]{clayton-et-al} similarly conjectured that none of the Hecke polynomial coefficients vanish in the level one case. We propose the following conjecture that further extends both the generalized Lehmer conjecture, and \cite[Conjecture 5.1]{clayton-et-al}. The results in this paper verify Conjecture \ref{conj:gengenlehmer} in all but finitely many cases.

\begin{conjecture} \label{conj:gengenlehmer}
    Fix integers $m \geq 1$ and $r \geq 1$. Then the $r$-th coefficient of the Hecke polynomial $T_m(N,k)(x)$ is nonvanishing for all $N \ge 1$ coprime to $m$, and $k = 12r$ or $\geq 12 r+4$ even.
\end{conjecture}

We note that these lower bounds on $k$ are the minimum possible. For any $k$ less than these bounds, we will have $\dim S_k(\Gamma_0(1)) < r$, and hence $c_r(m,1,k) = 0$, trivially. Even relaxing the lower bound on $k$ to just requiring that $\dim S_k(\Gamma_0(N)) \geq r$ will not work; Rouse \cite[Theorem 1.2]{rouse} showed that for any given $m$ and $k \in \{4,6,8,10,14\}$, $\Tr T_m(N,k) = 0$ for infinitely many $N$.

We now survey all the relevant previous results through the lens of Conjecture \ref{conj:gengenlehmer} (although they were not explicitly stated in these terms).
\begin{itemize}
    \item When $r\geq1$, $m=1$, Conjecture \ref{conj:gengenlehmer} follows from the fact that $\dim S_k(\Gamma_0(N)) \geq r$ for $k=12r$ and $k\geq 12r+4$.
    \item In 2006, when $r=1$ and $m$ is a non-square, Rouse \cite{rouse} showed Conjecture \ref{conj:gengenlehmer} for all but finitely many $k$, and for $100\%$ of $N$. 
    When $r=1, m=2$, he also completely verified Conjecture \ref{conj:gengenlehmer}. 
    \item In 2022, when $r=1$, $m=2$, Chiriac and Jorza  \cite{chiriac-jorza} verified Conjecture \ref{conj:gengenlehmer} in the case of $N=1$. 
    \item In 2023, when $r=2$, $m\geq2$, Clayton et al. \cite{clayton-et-al} showed Conjecture \ref{conj:gengenlehmer} in the case of $N=1$ for all but finitely many $k$. 
    When $r=2$, $m=2$, they also completely verified Conjecture \ref{conj:gengenlehmer}.  
    \item In 2023, when $r=1$, $m=3$, Chiriac et al.  \cite{chiriac} verified Conjecture \ref{conj:gengenlehmer} in the case of $N=1$. 
    \item In 2024, when $r=2$, $m\ge2$, we \cite{ross-xue}  showed Conjecture \ref{conj:gengenlehmer} for all but finitely many pairs $(N,k)$. When $r=2$, $m=3,4$, we also completely verified Conjecture \ref{conj:gengenlehmer}. 
    \item In 2024, when $r=2$, $m\ge2$, Cason et al. \cite{cason-et-al} showed a corresponding conjecture on the newspace $S_k^{\text{new}}(\Gamma_0(N))$ for all but finitely many pairs $(N,k)$. When $r=2,m=2,4$, they also completely verified the corresponding conjecture on the newspace. 
    \item In this paper, when $r\geq 1$ and $m$ is a square, Corollary \ref{cor:sgn-msq} proves Conjecture \ref{conj:gengenlehmer} for all but finitely many pairs $(N,k)$. 
    \item In this paper, when $r$ is even and $m$ is a non-square, Corollary \ref{cor:sgn-mnsq-even} proves Conjecture \ref{conj:gengenlehmer} for all but finitely many pairs $(N,k)$. 
    \item In this paper, when $r$ is odd and $m$ is a non-square, Corollary \ref{cor:sgn-mnsq-odd} shows that for $k$ fixed, if Conjecture \ref{conj:gengenlehmer} holds for $r=1$, then it also holds for each odd $r$ for all but finitely many $N$. In particular, combining with Rouse's result, this means that there exists a finite set $K$ such that: (1) for all $k \not\in K$, Conjecture \ref{conj:gengenlehmer} holds for all but finitely many $N$, and (2) even for $k \in K$, Conjecture \ref{conj:gengenlehmer} holds for $100\%$ of $N$.
\end{itemize}
We observe that the last result listed here essentially reduces the problem of studying odd-indexed coefficients $c_r$ to just studying the trace, $-c_1$.

\section*{Acknowledgements}
We would like to thank the anonymous referees for their helpful comments.

\bibliographystyle{plain}
\bibliography{bibliography.bib}

\end{document}